\documentclass[11pt, dvipsnames]{article}
\usepackage[margin=1in]{geometry}
\pdfoutput =1
\usepackage{setspace}
\usepackage[round,sort&compress]{natbib}
\let\cite\citep
\usepackage{amsmath}
\usepackage{amssymb}
\usepackage{bigfoot}

\usepackage{algorithm}
\usepackage{varwidth}
\usepackage{booktabs}
\usepackage{algpseudocode}
\allowdisplaybreaks
\usepackage{hyperref}
\makeatletter
\newcommand{\printfnsymbol}[1]{%
  \textsuperscript{\@fnsymbol{#1}}%
}
\makeatother
\usepackage{graphicx}
\usepackage{caption}
\usepackage{natbib}
\usepackage{subcaption}
\usepackage{color}
\usepackage{bbm}
\usepackage{float}

\usepackage{amsthm}
\usepackage{dsfont}
\usepackage{mathrsfs}
\usepackage{lmodern}
\usepackage{hyperref}

\usepackage{todonotes}
\usepackage{mathtools}
\usepackage{verbatim}
\usepackage{subcaption}
\usepackage{enumitem}
\usepackage{amsthm}
\usepackage{macros}

\usepackage{booktabs}

\usepackage{xcolor}

\usepackage{appendix}
\title{Understanding the Impact of Coalitions between EV Charging Stations
}
\author{
Sukanya Kudva\thanks{Equal contribution (alphabetical order).} 
\thanks{
Chinmay Maheshwari (chinmay\_maheshwari@berkeley.edu), Kshitij Kulkarni (kshitijkulkarni@berkeley.edu), and Shankar Sastry (shankar\_sastry@berkeley.edu) are with the EECS, UC Berkeley. 
Sukanya Kudva (sukanya\_kudva@berkeley.edu) and Anil Aswani (aaswani@berkeley.edu) are with IEOR, UC Berkeley.\\
This material is based upon work supported by the National Science Foundation under Grant No. DGE-2125913 and Grant No. CMMI-1847666, and Collaborative Research: Transferable, Hierarchical, Expressive, Optimal, Robust, Interpretable NETworks (THEORINET) under Grant No. DMS-2031899.} ,
Kshitij Kulkarni$^{*\dagger}$, Chinmay Maheshwari$^{*\dagger}$,\\ Anil Aswani$^{\dagger}$ and Shankar Sastry$^{\dagger}$
}
\date{}

\begin{document}
\maketitle

\begin{abstract}
The rapid growth of electric vehicles (EVs) is driving the expansion of charging infrastructure globally. As charging stations become ubiquitous, their substantial electricity consumption can influence grid operation and electricity pricing. Naturally, \textit{some} groups of charging stations, which could be jointly operated by a company, may coordinate to decide their charging profile. While coordination among all charging stations is ideal, it is unclear if coordination of some charging stations is better than no coordination. In this paper, we analyze this intermediate regime between no and full coordination of charging stations. We model EV charging as a non-cooperative aggregative game, where each station's cost is determined by both monetary payments tied to reactive electricity prices on the grid and its sensitivity to deviations from a desired charging profile. We consider a solution concept that we call $\collusion$-Nash equilibrium, which is tied to a coalition $\collusion$ of charging stations coordinating to reduce their costs. We provide sufficient conditions, in terms of the demand and sensitivity of charging stations, to determine when independent (aka uncoordinated) operation of charging stations could result in lower overall costs to charging stations, coalition and charging stations outside the coalition. Somewhat counter to common intuition, we show numerical instances where allowing charging stations to operate independently is better than coordinating a subset of stations as a coalition. Jointly, these results provide operators of charging stations insights into how to coordinate their charging behavior, and open several research directions.\end{abstract}
\newpage

\section{Introduction}
The proliferation of electric vehicles (EVs) has brought major changes to road transportation. EVs could play a significant role in the transition to a sustainable energy-based future and are projected to surpass traditional internal combustion engine-based vehicles in the coming decades\cite{bloomberg2021hitting}.  
The growth in EVs has led to the genesis of a new industry around building faster and more accessible charging infrastructure \cite{brown2024electric}, comprising several electric vehicles charging companies (EVCCs) such as Tesla, EVgo, and Chargepoint. 
One major challenge associated with the design of charging infrastructure is that charging stations, especially fast-charging stations, draw a considerable amount of electricity when in operation \cite{tesla, ElectrifyAmerica, porshe, EVGo} and may adversely impact the grid infrastructure \cite{ alexeenko2023achieving, lee2014electric, escudero2012charging}. The ideal solution to this challenge is to coordinate the demands from all charging stations to distribute the load on the grid \cite{alexeenko2023achieving, obeid2023learning}. 
However, this is hard to implement in practice due to the high communication and computational costs of such centralized approaches, and concerns about privacy of the information shared by EV charging stations\cite{brown2024electric}. 
Since coordination between all EV charging stations is impractical in most cases, we consider the scenario where only \emph{some} of these stations coordinate and form a \emph{coalition} (refer to Figure \ref{fig:schematic} for a schematic). While it might intuitively seem that some coordination is always better than no coordination, we give a counter-example to show this is not necessarily true. 
Particularly, we study the following question in this paper: 
\begin{quote}
   \textbf{Q:} \textit{
   Could coordination between a few potentially heterogeneous charging stations lead to undesirable consequences?
   }
\end{quote}

\begin{figure}
    \centering
    \includegraphics[width=0.5\textwidth]{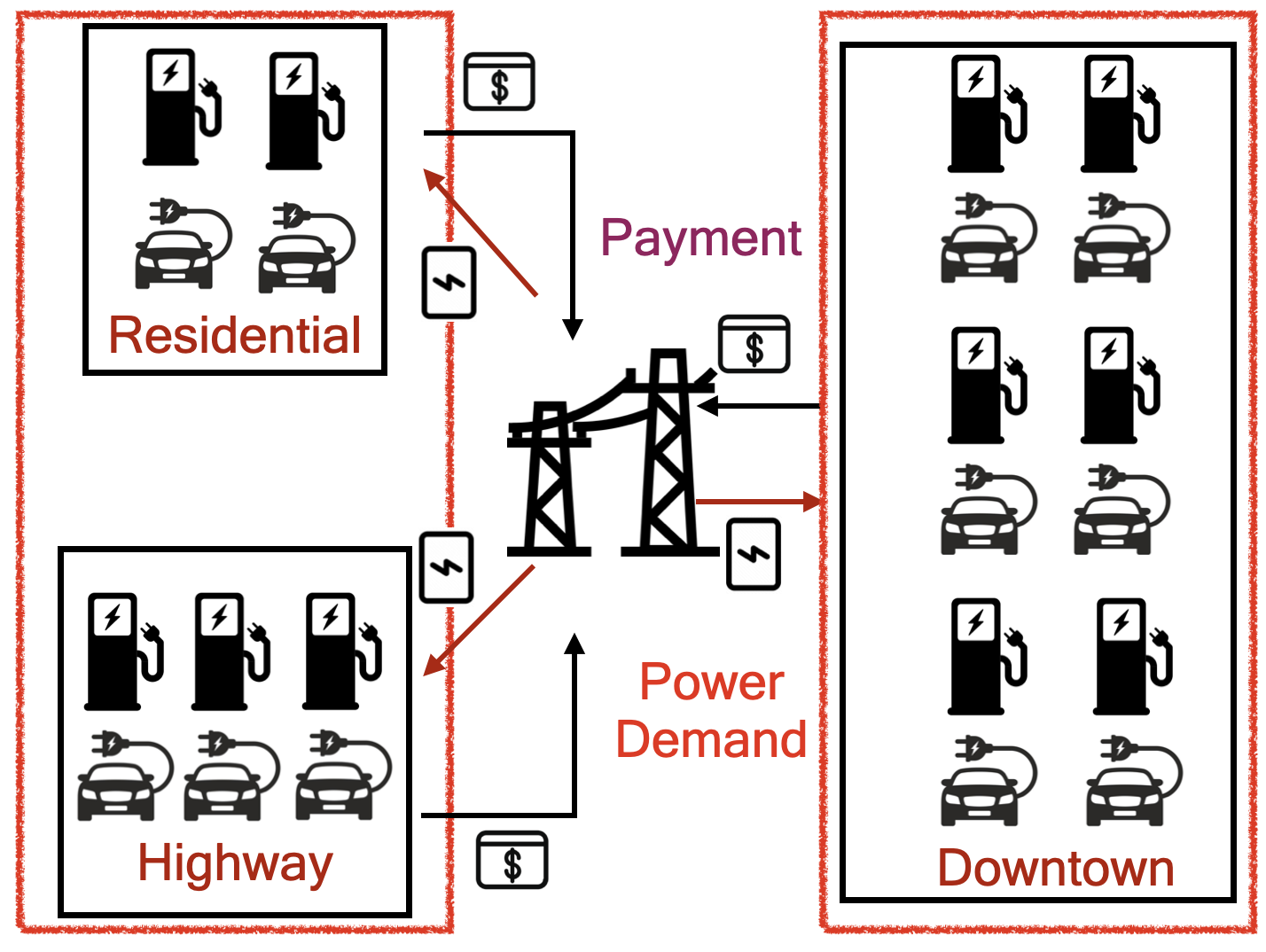}
    \caption{A schematic depiction of charging stations located in different areas (e.g. residential, downtown, highway) and owned by different EVCCs. Charging stations that are contained in the same red box are owned by same company (i.e. form a coalition). Each charging station demands charge from the grid, which determines the prices.}
    \label{fig:schematic}
\end{figure}
To answer this question, we model the interaction between charging station as a non-cooperative aggregative game. 
Charging stations are strategic agents that draw power from the grid over a finite time window, and have different location-specific charging demands and different sensitivities to deviations from a desired operating demand profile. 
Each charging station aims to minimize its cost comprising of \((i)\) a payment for power demanded from the grid, and \((ii)\) the deviations from their desired operating charging profile. In our model, we consider that the price per unit of power depends on the total power demanded by all charging stations (resulting in a game theoretic interaction between charging stations).  Therefore, in our model
charging stations act as ``price-makers", rather than ``price-takers" \cite{ma2011decentralized,gan2012optimal,paccagnan2018efficiency, deori2017connection}. 
Finally, some charging stations enter into a coalition (e.g. those owned by a single EVCC) to coordinate their total demand in order to minimize their total cost.

We compare the outcome of the game-theoretic interaction in two scenarios: \textit{(a)} when some charging stations form a coalition, 
and \textit{(b)} when each charging station operates independently without any coordination.
We differentiate the equilibrium in these two scenarios as $\collusion-$Nash and Nash equilibrium respectively, and characterize them analytically in Theorem \ref{thm: CollusionEquilibrium} and Corollary \ref{cor: NE} respectively. In general, we observe that the equilibrium decomposes into two components: a charging profile uniformly distributed across time and a correction term to account for the coalition of charging stations (Theorem \ref{thm: CollusionEquilibrium}).
We assess the scenarios \textit{(a)-(b)} in terms of three metrics depending on the overall cost experienced by: \(\textit{(i)}\) all charging stations (aka societal cost); \(\textit{(ii)}\) all charging stations within a coalition, and \(\textit{(iii)}\) all charging stations outside of coalition. We identify sufficient conditions on the game parameters under which the formation of a coalition will be worse than the independent operation of charging stations in terms of all three metrics \(\textit{(i)}-\textit{(iii)}\), as presented in Theorem \ref{thm: WelfareComparison} and \ref{theorem: uniform_demand_comparison}.  Furthermore, we present numerical examples that satisfy these conditions, demonstrating that coalitions can lead to worse outcomes for society, the coalition, and non-coalition charging stations alike. Notably, we find that these outcomes persist even when we relax the assumptions in our theoretical results, highlighting scenarios where the conventional intuition about coordinating charging stations does not hold.

\section{Related Works}
\textbf{Aggregative games and EV charging.} 
Several works have proposed EV charging as a non-cooperative game \cite{ma2011decentralized,gan2012optimal,paccagnan2018efficiency, deori2017connection}. These works mainly study existence, uniqueness, and computation of Nash equilibria of the EV charging game, and analyze these equilibria via measures such as the price of anarchy (PoA) \cite{paccagnan2018efficiency}. Our key differentiation from this line of literature is to understand the impact of the formation of a coalition of a subset of EV stations.

\textbf{Coalitions and equilibria.} 
Strong Nash equilibrium is an outcome wherein no coalition of agents can collectively deviate from their strategy to improve their utilities, and was first introduced as a concept in \cite{aumann1959acceptable}.
Since this seminal paper, numerous studies have delved into necessary and sufficient conditions for its existence \cite{nessah2014existence}, computational properties \cite{gatti2017verification}, and its performance across various game classes, often measured by the \emph{k-strong price of anarchy} \cite{ferguson2023collaborative, epstein2007strong, andelman2009strong, feldman2015unified, bachrach2014strong, chien2009strong}. This metric quantifies the worst-case welfare loss at strong equilibrium compared to optimal welfare. Unlike strong Nash equilibrium, our notion of \(\collusion-\)Nash equilibrium (Definition \ref{def: CE}) is a relaxation and only requires stability with respect to a specific coalition \(\collusion\) and not all possible coalitions. Strong Nash equilibrium may not always exist in our setting but we provide an explicit characterization of $\mathcal{C}$-Nash equilibrium (Theorem \ref{thm: CollusionEquilibrium}).

\textbf{Users as \emph{price-takers}.} Several studies have delved into the potential detrimental effects of uncoordinated Electric Vehicle (EV) charging on the power grid \cite{quiros2018electric, gruosso2016analysis}. In response, efforts such as those outlined in \cite{alexeenko2023achieving, obeid2023learning, fu2020intelligent} have tackled the challenge of devising incentive mechanisms to coordinate small users, often categorized as ``price-takers," to shift their charging windows and mitigate grid impacts. However, in contrast to these approaches, our work conceptualizes charging stations, which serve many EVs and thus can aggregate demand, as ``price-makers" within electricity markets.

\section{Model}
\paragraph*{Notations} For any positive integer \(m\), we define \([m]:= \{1,2,...,m\}\). Consider two matrices \(A = (a_{ij})_{i\in [m], j\in [n]}\in \R^{m\times n}, B \in \R^{p\times q}\). We define \(A\otimes B\in \R^{pm\times qn}\) to be the Kronecker product of matrices \(A\) and \(B\). 
For any finite set \(X\), we define \(\mathbf{1}_{X}\in \R^{|X|}\) to be a vector with all entries to be 1. For any vector \(x\in \R^m\) and \(p\subseteq [m]\), we use the notation \(x_{p}\in R^{|p|}\) to denote the components of vector \(x\) corresponding to \(p\). Additionally, we denote \(x_{-p}\in R^{m-|p|}\) to denote the components of vector \(x\) corresponding to all entries that are not in \(p\). 

\paragraph{Charging stations as strategic entities}
Consider a game comprising of \(\playerSet\)  charging stations, each operating as a strategic entity making decisions on their charging levels throughout a charging period spanning \(T\) units of time. Each station \(i\in [N]\) is characterized by a \emph{nominal charging profile} \((\bar{x}_i^t)_{t\in [T]}\), where the total charge demanded is \(d_i\). That is, \(\sum_{t\in [T]}\bar{x}_i^t = d_i\). The variation of charge demanded between stations reflect differences in the number of vehicles typically utilizing each charging facility.

The actual charging profile of any station may differ from the nominal charging profile due to externality imposed by other stations through electricity prices (to be defined shortly). We denote  \({x}_i^t\) to be the charge demanded by station \(i\) during hour \(t\). Define \(X_i = \{x_i\in \R^{T}: \sum_{t\in [T]}x_i^t = d_i\}\) to be the set of feasible charging profiles for station $i$ \footnote{We do not impose non-negativity constraints on the charging profile. This modeling decision is justified, as charging stations have the capability not only to draw power from the grid but also to inject power into the grid when necessary \cite{khan2019fast}. Correspondingly, we also do not impose upper bounds on the charging profiles; we assume that there is always enough charge available to meet the charging demand for a given problem instance. It is an interesting direction of future research to impose additional constraints on the set \(X\) which align more closely with real-world conditions. 
}. With a slight abuse of notation, we define 
\(x^t\in \R^{N}\) as a vector of charging profile of all stations at time step \(t\in [T]\), and \(x = (x_i^t)_{t\in T, i\in \playerSet}\) to be the joint charging profile of all stations. 

In our model, we consider reactive prices where the price of electricity depends on the total charge demanded. More formally, the cost of incurred by any station \(i\) under a joint strategy \(x\) is represented as: 
\begin{align}\label{eq: cost_player_i}
    c_i(x) = \sum_{t\in [T]} p^t(x) x_i^t + \frac{\mu_i}{2}\|x_i - \bar{x}_i\|^2,
\end{align}
where \((a)\) \(\mu_i>0\) is the sensitivity parameter of station \(i\); \((b)\) for every \(t\in [T]\), \(p^t(x)\) denotes the price per unit of electricity when the joint charging profile of all stations is \(x\). It is through this price signal that the charging profiles of other charging stations impose externality on any station. In what follows we assume that the price function is linear, i.e. \(p^t(x) = a^t+ b^t\mathbf{1}_{N}^\top x^t\) for some \(a^t, b^t > 0\) (\cite{vaya2015price, paccagnan2018efficiency, deori2017connection}). The parameters \(a^t\) represents the price fluctuations due to non-EV demand on the grid. In this paper, we assume that for all time \(t\in [T]\), \(b^t = b\) for some positive scalar \(b\in \R\). 

\begin{remark}
    The heterogeneity in sensitivity parameters $\mu_i$ between different stations can be ascribed to the geographic location (cf. Figure \ref{fig:schematic}). For instance, a station positioned in downtown can exhibit heightened sensitivity in meeting its demand requirements compared to one situated within a residential neighborhood. 
\end{remark}

Nash equilibrium, defined below, is widely used to characterize the interaction in these kinds of charging games \cite{ma2011decentralized,gan2012optimal,paccagnan2018efficiency, deori2017connection}. 
\begin{defn}\label{def: NE}
    A joint charge profile \(x^\ast\) is a \emph{Nash equilibrium} if \(c_i(x_i^\ast,x_{-i}^\ast) \leq c_i(x_i, x_{-i}^\ast)\) for every \(i\in [
N], x_i\in X_i\). 
\end{defn}

\paragraph{Coalition between charging stations}
Coalitions between charging stations can be easily facilitated by the emerging electric vehicle charging companies (EVCCs) (such as Chargepoint, Tesla etc.) who operate multiple charging stations. 
We model that every coalition jointly decides their charging profiles. For ease of exposition, we consider only single coalition\footnote{All results presented in this article can be extended to encompass scenarios involving multiple coalitions.}. 
The goal of the coalition is to choose \(x_{\collusion}\in \prod_{i\in \collusion}X_i\) that  minimizes the coalition's cumulative cost function \(
    c_{\collusion}(x) = \sum_{i\in \collusion}c_i(x).\)
Next, we introduce a natural extension of Definition \ref{def: NE} to examine the outcomes of interactions in the presence of a coalition.
\begin{defn}\label{def: CE}
    A joint charge profile \(x^\dagger\in X\) is a \emph{\(\collusion\)-Nash equilibrium} if \textit{(i)} \(\sum_{i\in \collusion} c_{i}(x^\dagger) \leq  \sum_{i\in \collusion} c_{i}(x'_{\collusion}, x^{\dagger}_{-\collusion}) \) for every \(x'_{\collusion}\in \prod_{j\in \collusion}X_j,\) and \textit{(ii)}  \(c_{i}(x^\dagger) \leq  c_{i}(x'_{i}, x^{\dagger}_{-i})\) for every \(x'_{i}\in X_i, i\not\in \collusion.\)
\end{defn}
When \(|\collusion|=1\), Definition \ref{def: NE} and \ref{def: CE} are equivalent.
For the rest of the article, we shall denote \(\collusion-\)Nash equilibrium by \(x^\dagger\) and Nash equilibrium when charging stations act autonomously by \(x^\ast\).

\section{Results}
In Section \ref{ssec: Analytical}, we analytically characterize the \(\collusion-\)Nash equilibrium in terms of game parameters. Next, in Section \ref{ssec: WhenCNashBetter}, we derive sufficient conditions when Nash equilibrium is preferred over \(\collusion-\)Nash equilibrium, in terms of overall cost experienced by all charging stations, by charging stations within the coalition, and the charging stations outside the coalition. Finally, in Section \ref{ssec: CoalitionBenefit},we provide numerical instances where coalition is worse than Nash equilibrium. 

\subsection{Analytical Characterization of \(\collusion-\)Nash equilibrium}\label{ssec: Analytical}

\begin{theorem}\label{thm: CollusionEquilibrium}
    For any arbitrary coalition \(\mathcal{C}\subset[N]\), when $b>0$ and $\mu_i > 0 \; \forall i\in \playerSet$, the \(\collusion-\)Nash equilibrium exists, is unique, and takes the following form
    \begin{align}\label{eq: ColEq}
         x^{\dagger t} =&\notag \frac{d}{T} + \sum_{t'\in [T]} \frac{T \delta^{tt'}-1}{T} \left( b(\mathbf{1}_{\playerSet}\mathbf{1}_{\playerSet}^\top + \textbf{C}) + \mu \right)^{-1} \cdot \\&\hspace{5cm}\cdot \left(\mu \bar{x}^{t'} - a^{t'}\mathbf{1}_{\playerSet}\right),
    \end{align}
    where $\delta^{tt'}$ is the Kronecker delta function ($\delta^{tt'} = 1$ when $t = t'$ and is $0$ otherwise), $\mu= \textsf{diag}([\mu_1, \cdots,\mu_{\playerSet}])$ and \(
         \textbf{C} = \begin{pmatrix}        \mathbf{1}_\collusion\mathbf{1}_\collusion^\top & 0 \\ 
        0 & I_{\playerSet\backslash \collusion}
    \end{pmatrix}.
    \)
\end{theorem}
\begin{proof}
Before presenting the proof, we define some useful notation. Define \( A =[a^1, a^2,...,a^T]^\top\) and the concatenated vector of charging profiles \(x^{\dagger} := [x_1^{1^{\dagger}},\cdots, x_{\playerSet}^{1^{\dagger}}, \cdots \cdots, x_1^{t^{\dagger}}, \cdots,x_{\playerSet}^{t^{\dagger}}]^\top\). 

First, we show that for any coalition $\mathcal{C}\subset[N]$, the resulting game is a strongly monotone game. To ensure this it is sufficient to verify that \((i)\) the strategy set is convex, and \((ii)\) the game Jacobian is positive definite \cite{facchinei2003finite}. The convexity of strategy sets hold because the strategy sets \(X_i\) are simplices. Next, to compute the game Jacobian, we define \( \mathcal{G}_i(x) =  \frac{\partial c_\collusion(x)}{\partial x_i} \) if \(i\in \collusion,\) and \(\mathcal{G}_i(x) =  \frac{\partial c_i(x)}{\partial x_i}\) if \(i\not \in \collusion\). 

It can be verified that the game Jacobian, \(J(x)\), is \(
    J(x) = \nabla \mathcal{G}(x) = \Theta,
\)
where \(
    \Theta :=  I_T\otimes \left(b(\mathbf{1}_{\playerSet}\mathbf{1}_{\playerSet}^\top+\textbf{C}) + \mu \right)\), which is guaranteed to be a positive definite matrix by Lemma \ref{prop:inv}. Thus, for any $\mathcal{C}$, the game is strongly monotone, and the equilibrium is unique.

We now introduce two optimization problems:
\begin{align*}
P_\mathcal{C}( x_{-\mathcal{C}}^{\dagger}): \min_{x_{\mathcal{C}}\in X_{\collusion}} & \; \sum_{j \in \mathcal{C}} c_j(x_{\mathcal{C}}, x_{-\mathcal{C}}^{\dagger})  \\
\text { s.t. } & \sum_{t'\in [T]}x_j^{t'} = d_j \quad \forall j \in \mathcal{C}, \\
\quad \forall \ i\in [N]\backslash \collusion, ~P_{i}(x^{\dagger}_{-i} ): \min_{x_i\in X_i} & \; c_i(x_i, x_{-i}^{\dagger})  \\
\text { s.t. } & \sum_{t'\in [T]}x_i^{t'} = d_i.  
\end{align*}
By definition of \(\collusion-\)Nash equilibrium, charging stations within $\mathcal{C}$ jointly solve the optimization problem $P_\mathcal{C}( x_{-\mathcal{C}}^{\dagger})$ when the charging profiles of other stations $x_{-\mathcal{C}}^{\dagger}$ is known. Similarly, each player $i \in [N] \backslash \mathcal{C}$ solves the optimization problem $P_{i}(x^{\dagger}_{-i} )$ when the charging profile of other players $x_{-i}^{\dagger}$ is known. Note that each of the optimization problems has linear constraints, and hence Karush–Kuhn–Tucker (KKT) conditions are necessary for optimality. We solve the KKT conditions of all these problems simultaneously to get a unique solution. Since we established that the equilibrium is unique, this unique solution to the KKT conditions is the \(\collusion-\)Nash equilibrium.

The KKT conditions can be written in a compact form as in (\ref{eq: MatrixEq}), where $\lambda = (\lambda_i)_{i\in [N]}$ are the Lagrange multipliers associated with the linear constraint of every charging station.
\begin{subequations}\label{eq: MatrixEq}
    \begin{align}
        &\Theta x^{\dagger} = (I_T \otimes \mu) \bar{x} - (I_T \otimes \mathbf{1}_{\playerSet}) A - (\mathbf{1}_T \otimes I_{\playerSet}) \lambda
    \\
       &( \mathbf{1}_T^\top \otimes I_\playerSet) x^{\dagger} = d,
    \end{align}
\end{subequations}
where $\Gamma :=(\mathbf{1}_T^\top \otimes I_{\playerSet})\Theta^{-1} (\mathbf{1}_T\otimes I_\playerSet)$. We prove in Lemma \ref{prop:inv} that $\Gamma$ and $\Theta$ are invertible. Using this fact and solving \eqref{eq: MatrixEq} by eliminating $\lambda$, we obtain a closed-form expression for the  \(\collusion-\)Nash equilibrium in (\ref{eqnsol}).
    \begin{align} \label{eqnsol}
    x^\dagger =& \Psi_1(\Gamma, \Theta) \left((I_T \otimes \mu) \bar{x} - (I_T \otimes \mathbf{1}_{\playerSet}) A\right) + \Psi_2(\Gamma, \Theta) d,
    \end{align}
with $\Psi_1(\Gamma, \Theta) = \left(I_{NT} - \Theta^{-1}(\mathbf{1}_T \otimes I_{\playerSet})\Gamma^{-1} (\mathbf{1}_T^\top \otimes I_{\playerSet}) \right)\Theta^{-1}$ and $\Psi_2(\Gamma, \Theta) = \Theta^{-1}(\mathbf{1}_T \otimes I_{\playerSet})\Gamma^{-1}$. Expanding the terms using the time index, it can be checked that equation (\ref{eqnsol}) is equivalent to (\ref{eq: ColEq}), completing the proof.
\end{proof}
Next, we present a technical result, used in proof of Theorem \ref{thm: CollusionEquilibrium}, the proof of which is deferred to Appendix \ref{app: proof_lemma}.
\begin{lemma} \label{prop:inv}
$\Theta$ and $\Gamma$ are positive definite and hence invertible matrices.
\end{lemma}

We now make some remarks about Theorem \ref{thm: CollusionEquilibrium}. 
\begin{remark} The closed-form equilibrium solution in (\ref{eq: ColEq}) is comprised of a fixed term and a correction term.
The fixed term represents a charging profile that is uniform across time. The second term is the correction to account for the aggregate effects of the coalition, demand and charging preferences of the charging stations, and exogenous non-EV demand. As expected, when the time-dependent factors $a^{t'}$ and $ \bar{x}^{t'}$ are constant across time, the correction term is zero and charging stations charge uniformly across time.

\end{remark}
\begin{remark}
    The result in Theorem \ref{thm: CollusionEquilibrium} extends to the setting of multiple (non-overlapping) coalitions \(\collusion_1, \collusion_2, ...\collusion_K\) by setting
     \begin{align*}
         \textbf{C} = \begin{pmatrix}        \mathbf{1}_{\collusion_1}\mathbf{1}_{\collusion_1}^\top & 0 & \cdots & 0 \\ 0& \mathbf{1}_{\collusion_2}\mathbf{1}_{\collusion_2}^\top & \cdots & 0 \\ \vdots & \vdots & \ddots  & \vdots\\
        \vdots & \vdots & \mathbf{1}_{\collusion_K}\mathbf{1}_{\collusion_K}^\top  & \vdots\\  0& 0& 0 & I_{\playerSet\backslash \cup_{i=1}^{K}\collusion_i}
    \end{pmatrix}.
    \end{align*}
    In fact, all theoretical results in this article can be directly extended to the multiple coalition case by using the above \(\textbf{C}\) matrix. For the sake of clear presentation, we shall always work with \(K=1\) in the following text. 
\end{remark}
We conclude this section by stating the following specialization of Theorem \ref{thm: CollusionEquilibrium} that characterizes Nash equilibrium. 
\begin{corollary}\label{cor: NE}
    The Nash equilibrium \(x^\ast\) takes the following form:    \begin{align*}
         x^{\ast t} =& \frac{d}{T} + \sum_{t'\in [T]} \frac{T \delta^{tt'}-1}{T} \left( b(\mathbf{1}_{\playerSet}\mathbf{1}_{\playerSet}^\top + I_N) + \mu \right)^{-1}\cdot \left(\mu \bar{x}^{t'} - a^{t'}\mathbf{1}_{\playerSet}\right),
    \end{align*}
    where \(I_N\in \mathbb{R}^{N\times N}\) is the identity matrix. 
\end{corollary}

Proof of Corollary \ref{cor: NE} follows by setting \(\collusion=\{1\}\) in Theorem \ref{thm: CollusionEquilibrium}.
\subsection{When is \(\collusion-\)Nash equilibrium beneficial?}\label{ssec: WhenCNashBetter}
In this subsection, we study conditions under which \(\collusion-\)Nash equilibrium will be preferred over Nash equilibrium and vice-versa. For a charging profile $x$, define the total cost incurred by a group of charging stations in $\mathcal{S} \subseteq [\playerSet]$ as
\begin{align}\label{eq:totalcost}
    c_{\mathcal{S}}(x) := \sum_{i\in \mathcal{S}}c_i(x), \quad \forall \ x\in X.
\end{align}
In order to compare the outcome under Nash equilibrium and \(\collusion-\)Nash equilibrium, we use the following metrics:
\begin{align}\label{eq: Eval-Metric}
\underbrace{\frac{c_{[\playerSet]}(x^\ast)}{c_{[\playerSet]}(x^\dagger)}}_{:=\textsf{M}_{[\playerSet]}}, \quad \underbrace{\frac{c_{\mathcal{C}}(x^\ast)}{c_{\mathcal{C}}(x^\dagger)}}_{:=\textsf{M}_{\mathcal{C}}}, \quad \underbrace{\frac{c_{[N]\backslash\collusion}(x^\ast)}{c_{[N]\backslash\collusion}(x^\dagger)}}_{:=\textsf{M}_{[N]\backslash\collusion}} \tag{Eval-Metric},
\end{align}
where, recall, \(x^\ast\) is the Nash equilibrium when there is no coalition, and \(x^\dagger\) is the \(\mathcal{C}\)-Nash equilibrium. 
These metrics are the ratio of the total cost experienced by different groups of players at Nash equilibrium and at \(\collusion-\)Nash equilibrium. 
In particular, \(\textsf{M}_{[N]}\) is this ratio for all players, \(\textsf{M}_{\collusion}\) is this ratio for players within coalition \(\collusion\), and \(\textsf{M}_{[N]\backslash\collusion}\) is this ratio of players outside of the coalition.
For any \(\mathcal{S}\in \{[N], \collusion, [N]\backslash\collusion\}\), if \(\textsf{M}_{\mathcal{S}}<1\) then Nash equilibrium is preferred, otherwise \(\collusion-\)Nash equilibrium is preferred by the coalition \(\mathcal{S}\). Next, we theoretically characterize these metrics in two cases:
\subsubsection{Case A: Exogeneous price fluctuations $a^t$ are uniform across time}
Here, we study the case when \(a^t = a \; \forall t \in[T]\) for some \(a\in \mathbb{R}\). This can happen when the price is influenced by a constant non-EV demand throughout the day. Under this setting, the equilibria \(x^\dagger\) and \(x^\ast\) are represented below
\begin{prop}\label{prop: Collusion_and_NE_same_a} 
  Suppose \(a^t = a^{t'}\) for all \(t, t
'\in [T]\). Then for every \(t\in [T]\),
\begin{align*}
x^{\ast t} &= \frac{d}{T} + \sum_{t'\in [T]} \frac{T \delta^{tt'}-1}{T} \left( b(\mathbf{1}_{\playerSet}\mathbf{1}_{\playerSet}^\top + I_N) + \mu \right)^{-1} \mu \bar{x}^{t'},\\
    x^{\dagger t} &= \frac{d}{T} + \sum_{t'\in [T]} \frac{T \delta^{tt'}-1}{T} \left( b(\mathbf{1}_{\playerSet}\mathbf{1}_{\playerSet}^\top + {C}) + \mu \right)^{-1} \mu \bar{x}^{t'}.
\end{align*}  
\end{prop}

{Further, consider the scenario when all the charging stations prepare for similar peak and low demand hours, and use charging rate recommendations from EV manufacturers to predict their demand requirements.} That is, they desire similar demand profiles (up to a constant factor) across the day. This scenario is captured in Assumption \ref{assm: DecomposeDesiredDemand}.
  
\begin{assm}\label{assm: DecomposeDesiredDemand}
  The {desired demand} of each charging station \( i\in [N]\) at time step \(t\in [T]\) is \(
    \bar{x}_i^t = d_i \alpha^t,\)
where \(\sum_{t\in [T]}\alpha^t = 1\) and \(\alpha^t\geq 0\) for all \(t\in [T]\). 
\end{assm}

{
Next, we delineate conditions under which the formation of a coalition is worse than the independent operation of charging stations in terms of verifiable conditions on game parameters. 
}

\begin{theorem}\label{thm: WelfareComparison}
   Suppose Assumption \ref{assm: DecomposeDesiredDemand} holds and \(a^t = a^{t'}\) for all \(t, t
'\in [T]\). Any {group} $\mathcal{S}\subseteq [N]$ incurs a lower cost in Nash equilibrium compared to \(\collusion-\)Nash equilibrium iff
\begin{align}\label{eq: Condition_a_unif}
    \sum_{i\in \mathcal{S}}f^\dagger_i(\mu,b,d) - f^{\ast}_i(\mu,b, d) \geq 0,
\end{align}
where for every \(i\in [N],\) \(f^{\dagger}_i(\mu,b,d) := \Delta^\dagger \Delta_i^\dagger+ \frac{\mu_i(\Delta_i^\dagger-d_i)^2}{2b}, f^{\ast}_i(\mu,b,d) :=\Delta^\ast \Delta_i^\ast + \frac{\mu_i(\Delta_i^\ast-d_i)^2}{2b}, {\Delta}_i^\dagger := \left((b(1_N1_N^\top+ \textbf{C})+\mu)^{-1}\mu d\right)_i,\) and \({\Delta}_i^\ast := \left((b(1_N1_N^\top+ I_N)+\mu)^{-1}\mu d\right)_i.\)
        Additionally, $\Delta^\dagger := \sum_{i\in [N]}\Delta^\dagger_i$ and $\Delta^\ast := \sum_{i\in [N]}\Delta_i^\ast$.
\end{theorem}
\begin{proof}
For a subset $\mathcal{S}$ of charging stations, the total cost under Nash equilibrium is lower than \(\collusion-\)Nash equilibrium if and only if
\begin{align}\label{eq: CompareWelfare}
        \sum_{i \in \mathcal{S}}c_i(x^\ast) \leq \sum_{i \in \mathcal{S}}c_i(x^\dagger).
    \end{align}

In the rest of the proof, we calculate these costs in terms of the game parameters. For every  \(t\in [T]\),  define \(
    F^t := T\alpha^t - \sum_{t'\in [T]}\alpha^{t'}.
\) Since $\sum_{t\in[T]} \alpha^t = 1$, it holds that \( F^t = T\alpha^t - 1 \) and \( \sum_{t\in [T]}F^t = 0 \). Using Assumption \ref{assm: DecomposeDesiredDemand} along with Proposition \ref{prop: Collusion_and_NE_same_a}, we get the following for \(\collusion-\)Nash equilibrium:
\begin{equation}
\begin{aligned}
    x^{\dagger t}_i =   \frac{d_i}{T}+ \frac{F^t}{T} \Delta_i^\dagger 
    \label{eq: xdag_unifa}, \quad \quad 
    \mathbf{1}_N^\top x^{\dagger t } = \frac{D}{T} + \frac{F^t}{T}\Delta^\dagger \\
    x^{\dagger t}_i - \bar{x}_i^t
    = \frac{F^t}{T} (\Delta_i^\dagger -d_i), 
\end{aligned}
\end{equation}
where   \(D = \sum_{i\in [N]}d_i\). 
Using \eqref{eq: cost_player_i} and \eqref{eq: xdag_unifa} and \(\sum_{t\in [T]}F^t = 0\), the cost of station \(i\) at \(\collusion-\)Nash equilibrium is:
\begin{align*}
c_i(x^\dagger) &=  ad_i + \left(b\sum_{t\in [T]}(1_N^\top x^{ t \dagger})x_i^{t \dagger}\right) + \frac{\mu_i}{2}\|x_i^\dagger-\bar{x}_i\|^2 \\ 
    &= ad_i + \frac{b}{T^2}\left(TDd_i+ \sum_{t\in [T]}(F^t)^2\Delta^\dagger\Delta_i^\dagger\right)
  \\& \quad \quad +  
    \frac{\mu_i}{2T^2}\sum_{t\in [T]} (F^t)^2(\Delta_i^\ast-d_i)^2.
\end{align*}

Consequently, the total cost of charging stations in $\mathcal{S}$ is:
\begin{align*}\label{eq: UniformAlphaCE}
\sum_{i\in \mathcal{S}}c_i(x^\dagger)
    &= \left(aD_{\mathcal{S}} + \frac{bDD_{\mathcal{S}}}{T}\right) \notag \\& + \frac{\mathcal{F}}{T^2}\left(\sum_{i\in \mathcal{S}}b\Delta^\dagger\Delta^\dagger_i+\frac{\mu^i}{2}(\Delta_i^\dagger-d_i)^2\right),
\end{align*}
where \(D_{\mathcal{S}} := \sum_{i\in \mathcal{S}}d_i\), \(\mathcal{F}:= \sum_{t\in [T]}(F^t)^2\). 
Analogously, we can also compute the total cost at Nash equilibrium. 
Using 
these results, we conclude \eqref{eq: CompareWelfare} is equivalent to 
    {
    \begin{align*}
        & aD_{\mathcal{S}} + \frac{bDD_{\mathcal{S}}}{T} + \frac{\mathcal{F}}{T^2}\left(\sum_{i\in \mathcal{S}}b\Delta^\dagger\Delta^\dagger_i+\frac{\mu^i}{2}(\Delta_i^\dagger-d_i)^2\right)\\&\geq  aD_{\mathcal{S}} + \frac{bDD_{\mathcal{S}}}{T} + \frac{\mathcal{F}}{T^2}\left(\sum_{i\in \mathcal{S}}b\Delta^\ast\Delta^\ast_i+\frac{\mu^i}{2}(\Delta_i^\ast-d_i)^2\right)\\
        & {\iff \sum_{i\in \mathcal{S}}f^\dagger_i(\mu,b,d) - f^{\ast}_i(\mu, b,d) \geq 0.}
       \end{align*}}
    \end{proof}
\begin{remark}    Interestingly, the condition \eqref{eq: Condition_a_unif} in Theorem \ref{thm: WelfareComparison} is independent of \(T, \alpha^t, a\) and depends only on \(\mu\) and \(d\). 
\end{remark}
\begin{corollary}
   When \(c_i(x^\dagger), c_i(x^\ast)\geq 0 \; \forall i \in [N]\), Theorem \ref{thm: WelfareComparison} can be used to analyze (\ref{eq: Eval-Metric}) as follows. For any \(\mathcal{S}  \in \{ [N], \mathcal{C}, \mathcal{[N] \backslash \mathcal{C}}\}\), \(\textsf{M}_{\mathcal{S}} \leq 1\) if and only if \(\sum_{i\in \mathcal{S}}f^\dagger_i(\mu,b, d) - f^{\ast}_i(\mu, b, d) \geq 0.\)
   \end{corollary}

\subsubsection{Case B: The desired demand \(\bar{x}^t\) is uniform across time}\label{ssec: UniformDemand}
In this subsection, we analyze the setting when \(\bar{x}^t\) is uniform in \(t\). This denotes a uniform spread of desired demand by the stations. 
{This resembles the nominal charging profile when the EV adoption will reach a critical mass when each charging station observes a constant flow of EV such that when averaged over any time window the nominal charge is uniformly spread.}

\begin{prop}\label{prop: Collusion_and_NE_same_x}
If \(\bar{x}^t = d/T, \quad \forall \ t\in [T]\) then 
\begin{align*}
x^{\ast t} &= \frac{d}{T} - \sum_{t'\in [T]} \frac{T \delta^{tt'}-1}{T} \left( b(\mathbf{1}_{\playerSet}\mathbf{1}_{\playerSet}^\top + I_N) + \mu \right)^{-1} \mathbf{1}_N a^{t'},\\
    x^{\dagger t} &= \frac{d}{T} - \sum_{t'\in [T]} \frac{T \delta^{tt'}-1}{T} \left( b(\mathbf{1}_{\playerSet}\mathbf{1}_{\playerSet}^\top + {C}) + \mu \right)^{-1} \mathbf{1}_N a^{t'}.
\end{align*}
\end{prop}

\begin{theorem}\label{theorem: uniform_demand_comparison}
    Suppose \(\bar{x}^t = d/T,\) for every \(t\in [T]\). 
   Any group \(\mathcal{S}\subseteq[N]\) incurs a lower cost in Nash equilibrium when compared to \(\collusion-\)Nash equilibrium if and only if\(
    g^\dagger_{\mathcal{S}}(\mu,b) - g^{\ast}_{\mathcal{S}}(\mu,b) \geq (\Gamma^\dagger - \Gamma^\ast)h(A),\)
    where for every \(i\in [N],\)
    \begin{align*}
    g^{\dagger}_\mathcal{S}(\mu,b) &:= \frac{b}{T}\sum_{i\in\mathcal{S}}\Gamma^\dagger\Gamma^\dagger_i+\frac{\mu^i}{2T}(\Gamma_i^\dagger)^2 \\
    g^{\ast}_{\mathcal{S}}(\mu,b) &:= \frac{b}{T}\sum_{i\in\mathcal{S}}\Gamma^\ast\Gamma^\ast_i+\frac{\mu^i}{2T}(\Gamma_i^\ast)^2\\
    h(A) &:= \frac{\sum_{t,t'\in [T]}a^ta^{t'}\left({T \delta^{tt'}-1}\right)}{\sum_{t\in [T]}\left(\sum_{t'\in [T]}a^{t'} \left({T \delta^{tt'}-1}\right)\right)^2}, \\
       {\Gamma}_i^\ast &:= \left((b(1_N1_N^\top+ I_N)+\mu)^{-1}\mathbf{1}_{N}\right)_i, \\ 
       \Gamma^\dagger_i &:= \left(\left( b(\mathbf{1}_{\playerSet}\mathbf{1}_{\playerSet}^\top + {C}) + \mu \right)^{-1} \mathbf{1}_{\playerSet}\right)_i. 
\end{align*}
Additionally, $\Gamma^\ast := \sum_{i\in [N]}\Gamma_i^\ast$ and $\Gamma^\dagger := \sum_{i\in [N]}\Gamma_i^\dagger$.
\end{theorem}

\begin{proof}
The proof is analogous to that of Theorem \ref{thm: WelfareComparison} and is deferred to Appendix.
\end{proof}
\begin{remark}
The condition in Theorem \ref{theorem: uniform_demand_comparison} does not depend on the heterogeneity of demand of charging stations $d$.
\end{remark}

\begin{corollary}
   When \(c_i(x^\dagger), c_i(x^\ast)\geq 0 \; \forall i \in [N]\), Theorem \ref{theorem: uniform_demand_comparison} can be used to analyze (\ref{eq: Eval-Metric}) as follows.
   \begin{align*} \textsf{M}_{\mathcal{S}} \leq 1  \iff g_{\mathcal{S}}^{\dagger}(\mu, b) - g_{\mathcal{S}}^{*}(\mu, b) \geq (\Gamma^{\dagger} - \Gamma^*) h(A) \\
   \forall \mathcal{S}  \in \{ [N], \mathcal{C}, \mathcal{[N] \backslash \mathcal{C}}\}. \end{align*}
\end{corollary}

\begin{figure}[!htbp]
  \centering
  \begin{minipage}[b]{0.6\textwidth}
    \includegraphics[width=\textwidth]{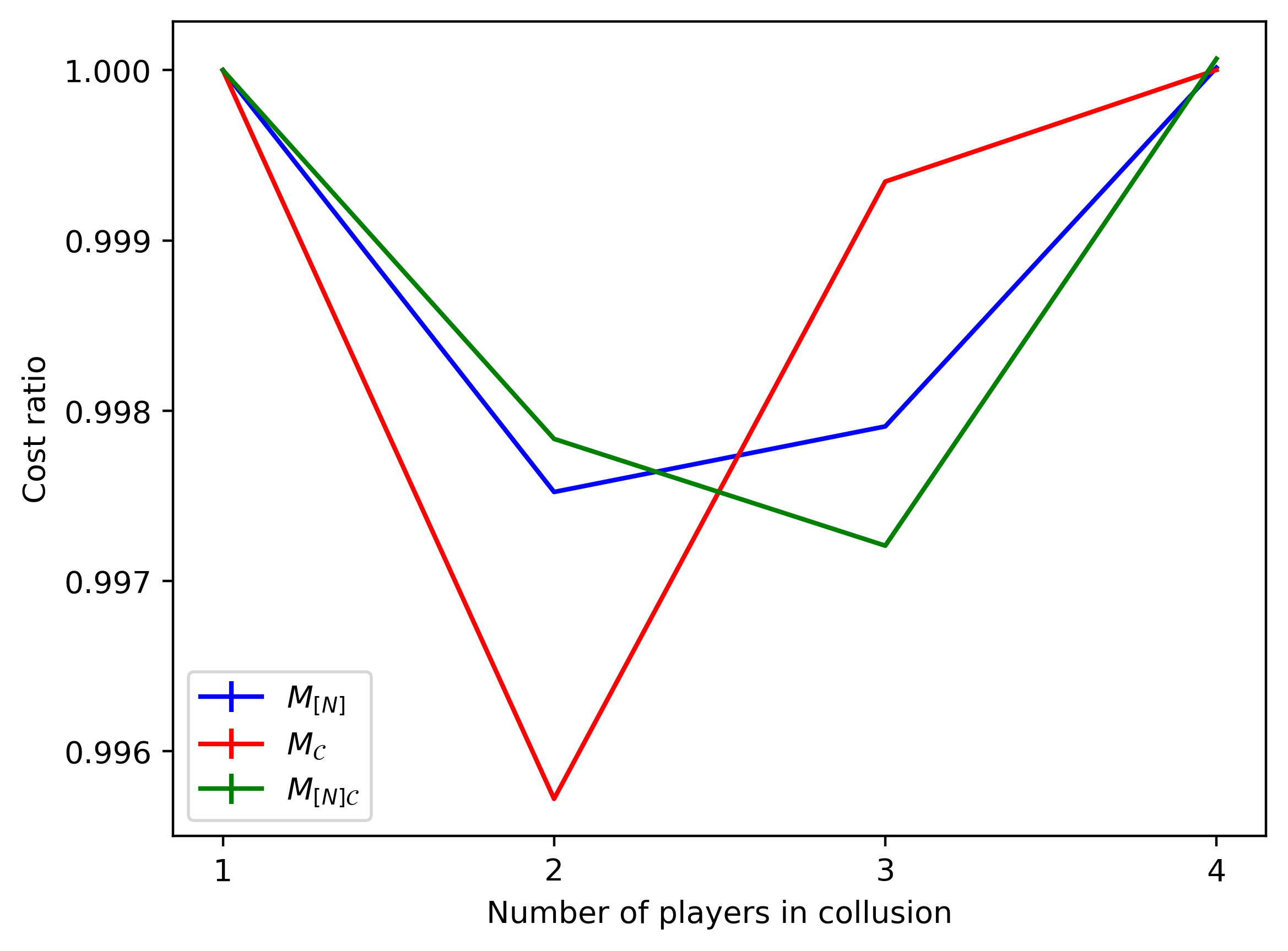}
    \caption{Setting \(\eta_1=0.4/T, \eta_2 = 1.6/T\) and \(\delta=0\)}
    \label{fig:figure1}
  \end{minipage}
  
  
  \begin{minipage}[b]{0.6\textwidth}
    \includegraphics[width=\textwidth]{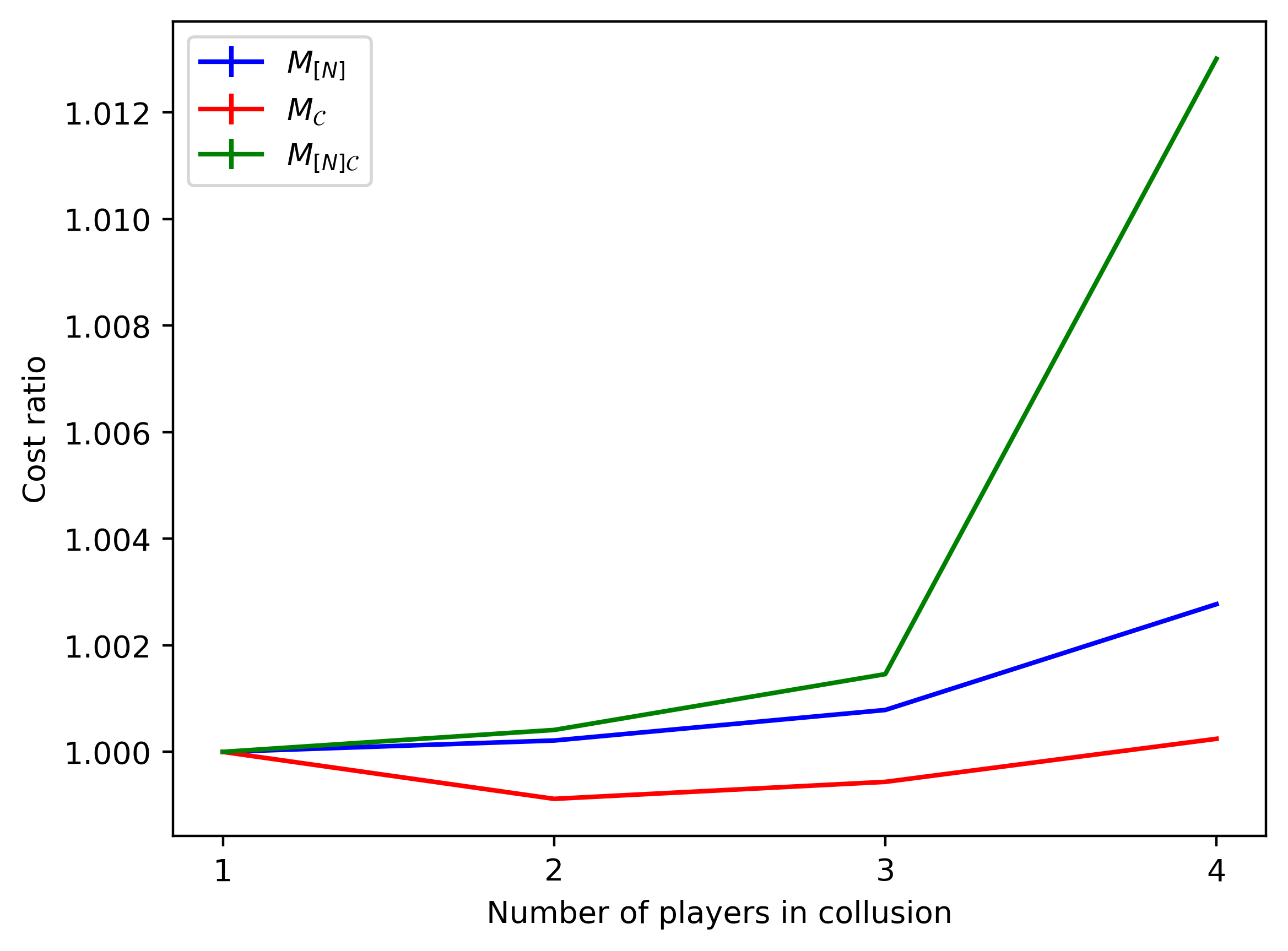}
    \caption{Setting \(\eta_1=1/T, \eta_2 = 1/T\) and \(\delta=0.4\)}
    \label{fig:figure2_a_uniform}
  \end{minipage}
\end{figure}

\subsection{Coalitions may not be always beneficial.}\label{ssec: CoalitionBenefit}
In this subsection, we construct instances where the formation of a coalition is not beneficial, which satisfy the setup discussed in Section \ref{ssec: WhenCNashBetter}.  We set \(N=5, T = 10, b=0.5\) and consider that each station can be of two types: \textsf{Type H}
 or \textsf{Type L}. For the purpose of this example, we consider that all stations in a coalition are of \(\textsf{Type L}\) and all stations outside the coalition are of \(\textsf{Type H}\). A station $i\in [N]$ is said to be \textsf{Type H} if \(d_i = 5\) and \(\mu_i = 1\), and of \textsf{Type L} if \(d_i = 1\) and \(\mu_i = 0.1\). Additionally, we set \(\alpha^t =\eta_1 \) for \(t\leq T/2 \), and \(\eta_2\) otherwise. Furthermore, we set \(a_t = 0.5 + \delta \nu_t \) where \(\nu_t\sim\textsf{Unif}([0,1])\).  
  We study the impact of the size of the coalition on various metrics presented in \eqref{eq: Eval-Metric}. Particularly, in Figure \ref{fig:figure1}, we study the case of \(\eta_1 = 0.4/T, \eta_2 = 1.6/T\) and \(\delta = 0\), which is aligned with the setup considered in Theorem \ref{thm: WelfareComparison}. Meanwhile, in Figure \ref{fig:figure2_a_uniform}, we study the case by setting \(\eta_1 =1/T, \eta_2 = 1/T\) and \(\delta = 0.4\), which is aligned with the setup considered in Theorem \ref{theorem: uniform_demand_comparison}. From these figures we conclude that there exist coalitions, where from the coalition's perspective (i.e. \(\textsf{M}_\collusion\)) the outcome under Nash equilibrium is preferred to the outcome under coordination. Furthermore, in Figure \ref{fig:figure1}, we find that there exist instances when the \(\collusion-\)Nash equilibrium is not preferred under every evaluation metric. Moreover, in Figure \ref{fig:figure2_a_uniform}, we find that the formation of coalition not only adversely impacts the coalition but also provides advantage to stations outside coalitions.

\section{Discussion}
In this section we expand on the experimental results from Section \ref{ssec: CoalitionBenefit}\footnote{The code to generate all figures in this section is available at \href{https://github.com/kkulk/coalition-ev}{https://github.com/kkulk/coalition-ev}} by relaxing the conditions imposed on game parameters.
\paragraph{Impact of simultaneous variations in \(a^t\) and \(\alpha^t\)} 
 In Figure \ref{fig:enter-label}, we study the impact of the size of the coalition in terms of \eqref{eq: Eval-Metric}. In contrast to Section \ref{ssec: CoalitionBenefit}, we consider both \(a^t\) and \(\alpha^t\) to be non-uniform and randomly assign stations to be of \textsf{Type H} with probability 0.2 if they are within the coalition. All stations outside of the coalition are assigned to be \textsf{Type L}. We find that if the size of coalition is small then Nash equilibrium turns out to be favorable along all metrics \eqref{eq: Eval-Metric}. However, as the size of coalition increases then it may be beneficial to form a coalition.

\begin{figure}
    \centering
    \includegraphics[width=0.6\textwidth]{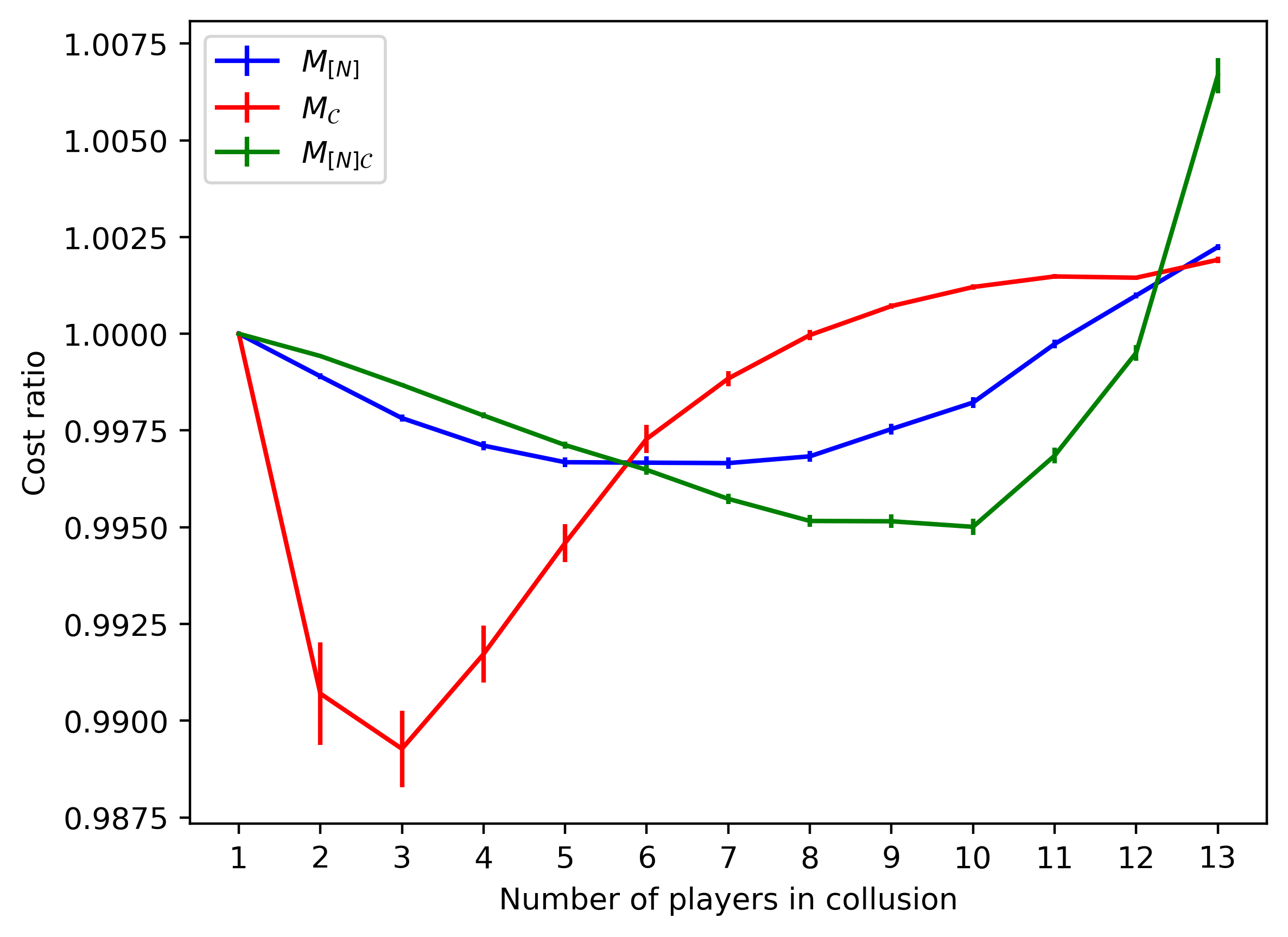}
    \caption{Comparison between Nash equilibrium and \(\collusion-\)Nash equilibrium with respect to \eqref{eq: Eval-Metric} under different size of coalition.}
    \label{fig:enter-label}
\end{figure}

\paragraph{Impact of the composition of coalitions} 
Here we show that not only the size of the coalition, but also the composition of the coalition, plays an important role in deciding whether the coalition is beneficial. To illustrate this point, we examine an example with \(N=3, T =10, b= 1\) and \(a^t = 10\) for all \(t\in [T]\). We posit a scenario where the first two stations form a coalition. Each station is one of two types, namely \textsf{Type H}
 or \textsf{Type L}. We call a station $i\in [N]$ to be of \textsf{Type H} if \(d_i = 5\) and \(\mu_i = 5\), and of \textsf{Type L} if \(d_i = 1\) and \(\mu_i = 1\). In Figure \ref{fig: ComparisonTables}, we present a comparative analysis, evaluating how different coalitions perform in terms of metric presented in \eqref{eq: Eval-Metric}. 
We find under some circumstances \(\textsf{M}_{\mathcal{S}} > 1\) for all \(\mathcal{S}\in \{[N], \collusion, [N]\backslash[\collusion]\}\). For instance, this is the case when the coalition is comprised of atleast one \textsf{Type H} station and the station outside coalition is of \textsf{TypeH}. On the contrary, perhaps surprisingly, there are also instances when \(\textsf{M}_{\mathcal{S}} < 1\) for all \(\mathcal{S}\in \{[N], \collusion, [N]\backslash[\collusion]\}\). For instance, this is the case when the stations within a coalition are of \textsf{Type H} and the station outside the coalition is of \textsf{Type L}.

\begin{figure}
\centering
\begin{subfigure}[b]{0.5\textwidth}
   \includegraphics[width=1\linewidth]{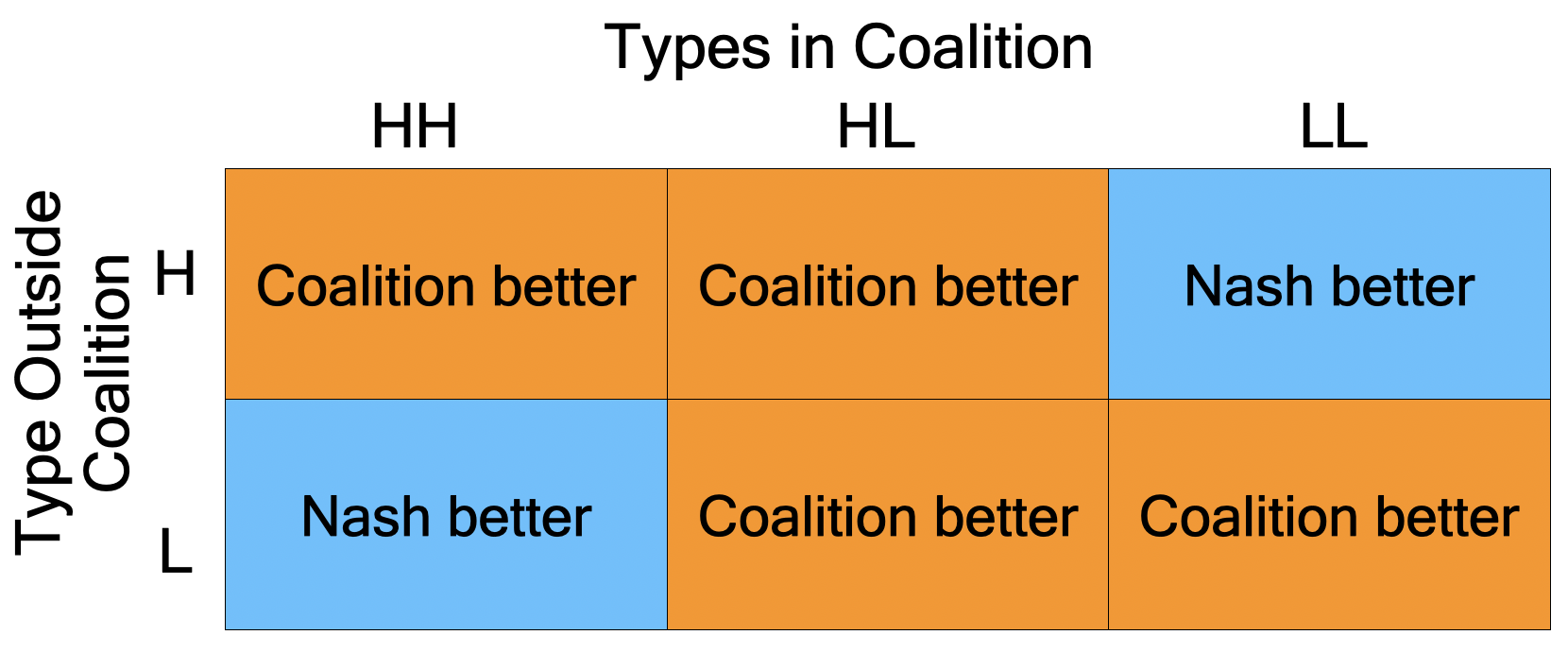}
   \caption{The phrase ``Coalition better'' (resp. ``Nash better'') implies \(\textsf{M}_{[N]} > 1\) (resp. \(\textsf{M}_{[N]} < 1\)).}
   \label{fig:WelfareComparison} 
\end{subfigure}

\begin{subfigure}[b]{0.5\textwidth}
   \includegraphics[width=1\linewidth]{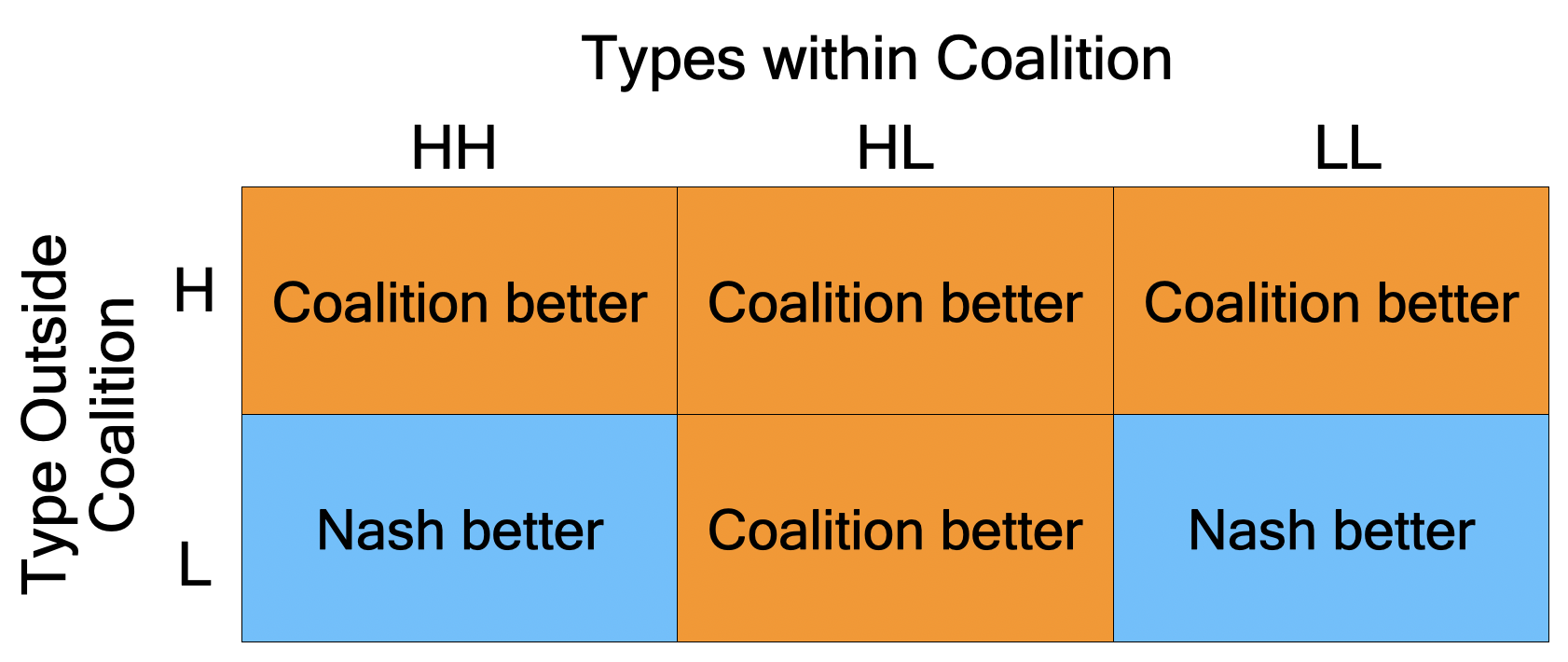}
   \caption{The phrase ``Coalition better'' (resp. ``Nash better'') implies \(\textsf{M}_{\collusion} > 1\) (resp. \(\textsf{M}_{\collusion} < 1\)).}
   \label{fig:CollusionComparison}
\end{subfigure}

\begin{subfigure}[b]{0.5\textwidth}
   \includegraphics[width=1\linewidth]{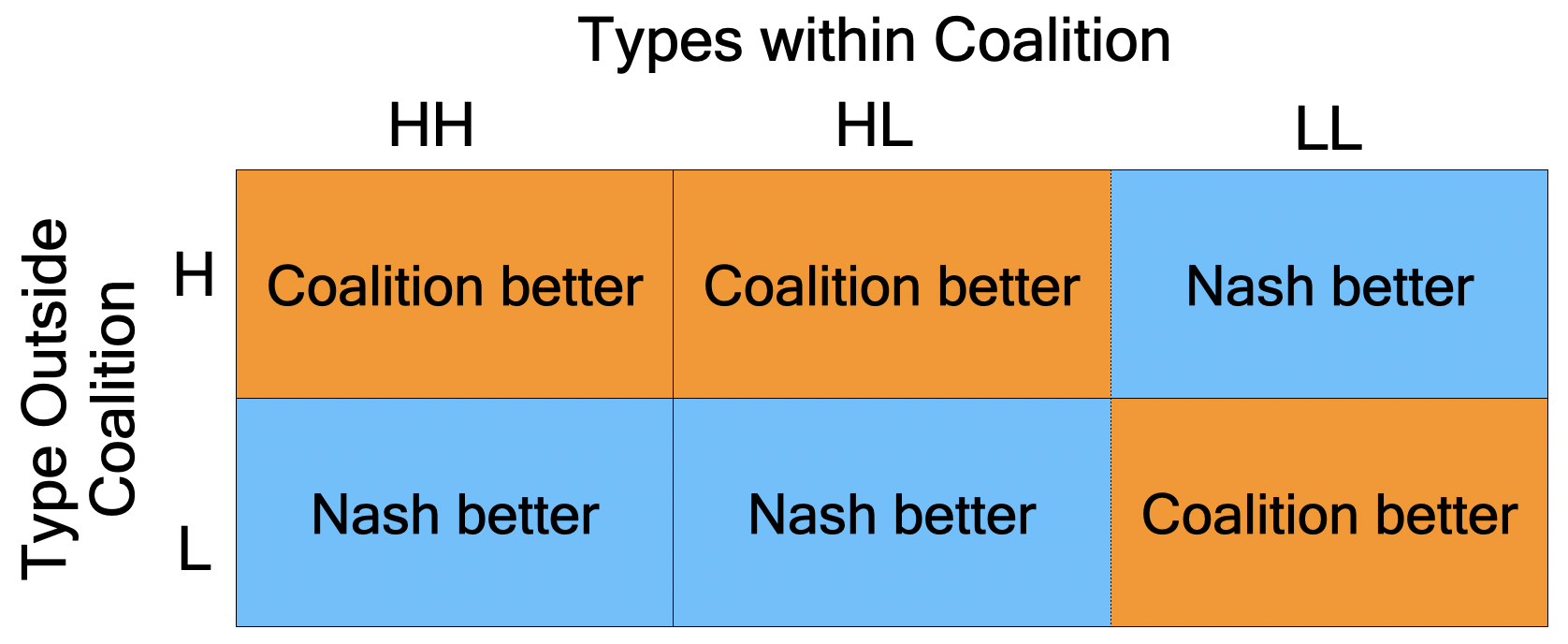}
   \caption{The phrase ``Coalition better'' (resp. ``Nash better'') implies \(\textsf{M}_{[N]\backslash\collusion} > 1\) (resp. \(\textsf{M}_{[N]\backslash\collusion} < 1\)).}
   \label{fig:NonCollusionComparison}
\end{subfigure}
\caption{Comparison of societal (overall) cost, coalitional cost, and the cost of the station outside the coalition in the three station game.}
\label{fig: ComparisonTables}
\end{figure}

\section{Conclusion}
In this work, we initiate a study to understand the impact of coalitions between charging stations as charging infrastructure continues to grow in coming years. As charging stations draw a substantial amount of electricity from the grid, they will be ``price-makers" on the electricity grid. More often than not, multiple charging stations are operated by same electric vehicle charging company (EVCC), which could facilitate coordination between charging stations. In this work, we analytically characterize the equilibrium outcome in the presence of coalitions. Our analysis hints at potential losses encountered by EVCCs if they coordinate all charging stations owned by them in the presence of heterogeneity in charging demand and user preferences. 

There are several interesting questions for future research in understanding the impact of coalitions between charging stations. In the current model we assume the overall demand of a charging station is fixed, but this demand could be affected by electricity prices \cite{lee2014electric, escudero2012charging, yuan2015competitive}. We also assume the price functions are linear; future work may extend this to generic nonlinear functions. Furthermore, there are additional operational constraints such as limited capacity of the energy infrastructure and bounds on charging rates of EVs that could be accounted for while computing the equilibrium.

\bibliography{refs}

\begin{thebibliography}{30}
\providecommand{\natexlab}[1]{#1}
\providecommand{\url}[1]{\texttt{#1}}
\expandafter\ifx\csname urlstyle\endcsname\relax
  \providecommand{\doi}[1]{doi: #1}\else
  \providecommand{\doi}{doi: \begingroup \urlstyle{rm}\Url}\fi

\bibitem[Alexeenko and Bitar(2023)]{alexeenko2023achieving}
Polina Alexeenko and Eilyan Bitar.
\newblock Achieving reliable coordination of residential plug-in electric vehicle charging: A pilot study.
\newblock \emph{Transportation Research Part D: Transport and Environment}, 118:\penalty0 103658, 2023.

\bibitem[America(2021)]{ElectrifyAmerica}
Electrify America.
\newblock How electric vehicle (ev) charging works.
\newblock https://www.electrifyamerica.com/how-ev-charging-works/, 2021.

\bibitem[Andelman et~al.(2009)Andelman, Feldman, and Mansour]{andelman2009strong}
Nir Andelman, Michal Feldman, and Yishay Mansour.
\newblock Strong price of anarchy.
\newblock \emph{Games and Economic Behavior}, 65\penalty0 (2):\penalty0 289--317, 2009.

\bibitem[Aumann(1959)]{aumann1959acceptable}
Robert~J Aumann.
\newblock Acceptable points in general cooperative n-person games.
\newblock \emph{Contributions to the Theory of Games}, 4\penalty0 (40):\penalty0 287--324, 1959.

\bibitem[Bachrach et~al.(2014)Bachrach, Syrgkanis, Tardos, and Vojnovi{\'c}]{bachrach2014strong}
Yoram Bachrach, Vasilis Syrgkanis, {\'E}va Tardos, and Milan Vojnovi{\'c}.
\newblock Strong price of anarchy, utility games and coalitional dynamics.
\newblock In \emph{Algorithmic Game Theory: 7th International Symposium, SAGT 2014, Haifa, Israel, September 30--October 2, 2014. Proceedings 7}, pages 218--230. Springer, 2014.

\bibitem[Bloomberg(2021)]{bloomberg2021hitting}
NEF Bloomberg.
\newblock Hitting the ev inflection point.
\newblock \emph{Electric Vehicle Price Parity Phasing Out Combustion Vehicle Sales in Europe; Bloomberg: New York, NY, USA}, 2021.

\bibitem[Brown et~al.(2024)Brown, Cappellucci, Heinrich, and Cost]{brown2024electric}
Abby Brown, Jeff Cappellucci, Alexia Heinrich, and Emma Cost.
\newblock Electric vehicle charging infrastructure trends from the alternative fueling station locator: Third quarter 2023.
\newblock Technical report, National Renewable Energy Laboratory (NREL), Golden, CO (United States), 2024.

\bibitem[Chien and Sinclair(2009)]{chien2009strong}
Steve Chien and Alistair Sinclair.
\newblock Strong and pareto price of anarchy in congestion games.
\newblock In \emph{International Colloquium on Automata, Languages, and Programming}, pages 279--291. Springer, 2009.

\bibitem[Deori et~al.(2017)Deori, Margellos, and Prandini]{deori2017connection}
Luca Deori, Kostas Margellos, and Maria Prandini.
\newblock On the connection between nash equilibria and social optima in electric vehicle charging control games.
\newblock \emph{IFAC-PapersOnLine}, 50\penalty0 (1):\penalty0 14320--14325, 2017.

\bibitem[Epstein et~al.(2007)Epstein, Feldman, and Mansour]{epstein2007strong}
Amir Epstein, Michal Feldman, and Yishay Mansour.
\newblock Strong equilibrium in cost sharing connection games.
\newblock In \emph{Proceedings of the 8th ACM conference on Electronic commerce}, pages 84--92, 2007.

\bibitem[Escudero-Garzas and Seco-Granados(2012)]{escudero2012charging}
J~Joaquin Escudero-Garzas and Gonzalo Seco-Granados.
\newblock Charging station selection optimization for plug-in electric vehicles: An oligopolistic game-theoretic framework.
\newblock In \emph{2012 IEEE PES Innovative Smart Grid Technologies (ISGT)}, pages 1--8. IEEE, 2012.

\bibitem[EVgo(2021)]{EVGo}
EVgo.
\newblock What is evgo | electric vehicle (ev) fast charging stations.
\newblock https://www.evgo.com/ev-drivers/the-evgo-network/, 2021.

\bibitem[Facchinei and Pang(2003)]{facchinei2003finite}
Francisco Facchinei and Jong-Shi Pang.
\newblock \emph{Finite-dimensional variational inequalities and complementarity problems}.
\newblock Springer, 2003.

\bibitem[Feldman and Friedler(2015)]{feldman2015unified}
Michal Feldman and Ophir Friedler.
\newblock A unified framework for strong price of anarchy in clustering games.
\newblock In \emph{International Colloquium on Automata, Languages, and Programming}, pages 601--613. Springer, 2015.

\bibitem[Ferguson et~al.(2023)Ferguson, Paccagnan, Pradelski, and Marden]{ferguson2023collaborative}
Bryce~L Ferguson, Dario Paccagnan, Bary~SR Pradelski, and Jason~R Marden.
\newblock Collaborative coalitions in multi-agent systems: Quantifying the strong price of anarchy for resource allocation games.
\newblock In \emph{2023 62nd IEEE Conference on Decision and Control (CDC)}, pages 3238--3243. IEEE, 2023.

\bibitem[Fu et~al.(2020)Fu, Dong, and Ju]{fu2020intelligent}
Zhengtang Fu, Peiwu Dong, and Yanbing Ju.
\newblock An intelligent electric vehicle charging system for new energy companies based on consortium blockchain.
\newblock \emph{Journal of Cleaner Production}, 261:\penalty0 121219, 2020.

\bibitem[Gan et~al.(2012)Gan, Topcu, and Low]{gan2012optimal}
Lingwen Gan, Ufuk Topcu, and Steven~H Low.
\newblock Optimal decentralized protocol for electric vehicle charging.
\newblock \emph{IEEE Transactions on Power Systems}, 28\penalty0 (2):\penalty0 940--951, 2012.

\bibitem[Gatti et~al.(2017)Gatti, Rocco, and Sandholm]{gatti2017verification}
Nicola Gatti, Marco Rocco, and Tuomas Sandholm.
\newblock On the verification and computation of strong nash equilibrium.
\newblock \emph{arXiv preprint arXiv:1711.06318}, 2017.

\bibitem[Gruosso(2016)]{gruosso2016analysis}
Giambattista Gruosso.
\newblock Analysis of impact of electrical vehicle charging on low voltage power grid.
\newblock In \emph{2016 International Conference on Electrical Systems for Aircraft, Railway, Ship Propulsion and Road Vehicles \& International Transportation Electrification Conference (ESARS-ITEC)}, pages 1--6. IEEE, 2016.

\bibitem[Khan et~al.(2019)Khan, Ahmad, and Alam]{khan2019fast}
Wajahat Khan, Furkan Ahmad, and Mohammad~Saad Alam.
\newblock Fast ev charging station integration with grid ensuring optimal and quality power exchange.
\newblock \emph{Engineering Science and Technology, an International Journal}, 22\penalty0 (1):\penalty0 143--152, 2019.

\bibitem[Lee et~al.(2014)Lee, Xiang, Schober, and Wong]{lee2014electric}
Woongsup Lee, Lin Xiang, Robert Schober, and Vincent~WS Wong.
\newblock Electric vehicle charging stations with renewable power generators: A game theoretical analysis.
\newblock \emph{IEEE transactions on smart grid}, 6\penalty0 (2):\penalty0 608--617, 2014.

\bibitem[Ma et~al.(2011)Ma, Callaway, and Hiskens]{ma2011decentralized}
Zhongjing Ma, Duncan~S Callaway, and Ian~A Hiskens.
\newblock Decentralized charging control of large populations of plug-in electric vehicles.
\newblock \emph{IEEE Transactions on control systems technology}, 21\penalty0 (1):\penalty0 67--78, 2011.

\bibitem[Nessah and Tian(2014)]{nessah2014existence}
Rabia Nessah and Guoqiang Tian.
\newblock On the existence of strong nash equilibria.
\newblock \emph{Journal of Mathematical Analysis and Applications}, 414\penalty0 (2):\penalty0 871--885, 2014.

\bibitem[Obeid et~al.(2023)Obeid, Ozturk, Zeng, and Moura]{obeid2023learning}
Hassan Obeid, Ayse~Tugba Ozturk, Wente Zeng, and Scott~J Moura.
\newblock Learning and optimizing charging behavior at pev charging stations: Randomized pricing experiments, and joint power and price optimization.
\newblock \emph{Applied Energy}, 351:\penalty0 121862, 2023.

\bibitem[Paccagnan et~al.(2018)Paccagnan, Parise, and Lygeros]{paccagnan2018efficiency}
Dario Paccagnan, Francesca Parise, and John Lygeros.
\newblock On the efficiency of nash equilibria in aggregative charging games.
\newblock \emph{IEEE control systems letters}, 2\penalty0 (4):\penalty0 629--634, 2018.

\bibitem[Quiros-Tortos et~al.(2018)Quiros-Tortos, Ochoa, and Butler]{quiros2018electric}
Jairo Quiros-Tortos, Luis Ochoa, and Timothy Butler.
\newblock How electric vehicles and the grid work together: Lessons learned from one of the largest electric vehicle trials in the world.
\newblock \emph{IEEE Power and Energy Magazine}, 16\penalty0 (6):\penalty0 64--76, 2018.

\bibitem[Schacke(2004)]{schacke2004kronecker}
Kathrin Schacke.
\newblock On the kronecker product.
\newblock \emph{Master's thesis, University of Waterloo}, 2004.

\bibitem[Tesla(2019)]{tesla}
Tesla.
\newblock Introducing v3 supercharging.
\newblock https://www.tesla.com/blog/introducing-v3-supercharging, 2019.

\bibitem[Tesla(2021)]{porshe}
Tesla.
\newblock Voelcker j. porsche claims it can double tesla’s fast-charging rate.
\newblock https://spectrum.ieee.org/transportation/efficiency/porsche-claims-it-can-double-teslas-fastcharging-rate, 2021.

\bibitem[Yuan et~al.(2015)Yuan, Huang, and Zhang]{yuan2015competitive}
Wei Yuan, Jianwei Huang, and Ying Jun~Angela Zhang.
\newblock Competitive charging station pricing for plug-in electric vehicles.
\newblock \emph{IEEE Transactions on Smart Grid}, 8\penalty0 (2):\penalty0 627--639, 2015.

\end{thebibliography}
\bibliographystyle{plainnat}
\newpage
\appendix
\section{Excluded proofs from main text}
\subsection{Proof of Lemma \ref{prop:inv}}\label{app: proof_lemma}
Note that \(\Theta =  I_T\otimes \left(b(\mathbf{1}_{\playerSet}\mathbf{1}_{\playerSet}^\top+\textbf{C}) + \mu \right)\) is positive definite if and only if each of the terms in the Kronecker product is positive definite \cite{schacke2004kronecker}. Hence, it is sufficient to show that $Q := b(\mathbf{1}_{N}\mathbf{1}_N^\top + \textbf{C}) + \mu$
is a positive definite matrix. 
First, note that $Q$ is a symmetric matrix. Additionally, for any \(z\in \R^N\), it holds that 
\begin{align*}
    z^\top Qz  
    =& \; b (\mathbf{1}_N^\top z)^2 + b(\sum_{i\in\collusion}z_i)^2 + \sum_{i\in \collusion}\mu_iz_i^2  +\sum_{i\in N\backslash\collusion}(b+\mu_i)z_i^2, 
\end{align*}
which is strictly positive for all \(z\neq 0\). 

Next, we show that \(\Gamma\) is invertible by calculating its inverse. Indeed, 
\begin{align*}
    \Gamma^{-1} =& \left((\mathbf{1}_T^\top \otimes I_{\playerSet})\Theta^{-1} (\mathbf{1}_T\otimes I_\playerSet)\right)^{-1} \\
    =& \left( \mathbf{1}_T^\top \mathbf{1}_T \otimes \left( b(\mathbf{1}_{\playerSet}\mathbf{1}_{\playerSet}^\top + C) + \mu \right)^{-1} \right)^{-1} \\
    =& \frac{1}{T}\left( b(\mathbf{1}_{\playerSet}\mathbf{1}_{\playerSet}^\top + C) + \mu \right) =\frac{1}{T} Q.
\end{align*}
This completes the proof. 

\subsection{Proof of Theorem \ref{theorem: uniform_demand_comparison}}
    For every \(t\in [T]\), define \(G^t := \sum_{t'\in [T]}a^{t'} \left({T \delta^{tt'}-1}\right)\). Using this, the expression for $x^{\dagger t}$ from Proposition \ref{prop: Collusion_and_NE_same_x} can be re-written as follows
    \begin{align}
        x^{\dagger t} =\frac{d}{T} -  \frac{1}{T}\left( b(\mathbf{1}_{\playerSet}\mathbf{1}_{\playerSet}^\top + {C}) + \mu \right)^{-1} \mathbf{1}_{\playerSet} G^t.
    \end{align}
Consequently, it holds that  
\begin{equation}
\begin{aligned}
x^{\dagger t}_i &= \frac{d_i}{T} - \frac{G^t}{T}\Gamma^\dagger_i\\ 
    \mathbf{1}_N^\top x^{\dagger t} &= \frac{D}{T} - \frac{G^t}{T}\Gamma^\dagger \\ x^{\dagger t}_i - \bar{x}^t_i &= -\frac{1}{T}\Gamma^\dagger_i G^t. 
\end{aligned}
\end{equation}

Consequently, combining previous equations with \eqref{eq: cost_player_i}, for every \(i\in [N]\), 
\begin{align}\label{eq: IndividualCost_xbar_uniform}
&c_i(x^\dagger) 
    = \sum_{t\in [T]}a^tx^{\dagger t}_i + \frac{b}{T^2}
    \sum_{t\in [T]} \left(D-G^t\Gamma^\dagger \right)\left( d_i - G^t \Gamma^\dagger_i \right)
    \notag \\ 
    &\quad \quad + 
    \frac{\mu_i (\Gamma_i^\dagger)^2}{2T^2}\sum_{t\in [T]} (G^t)^2\notag \\ 
    &= \sum_{t\in [T]}a^t\left(\frac{d_i}{T} - \frac{G^t}{T}\Gamma^\dagger_i\right)+ \frac{b}{T^2}\left(TDd_i+ \Gamma^\dagger\Gamma_i^\dagger\sum_{t\in [T]}(G^t)^2\right)
  \notag \\& \quad \quad + 
    \frac{\mu_i(\Gamma^\dagger_i)^2}{2T^2}\sum_{t\in [T]} (G^t)^2,
\end{align}
where the last equality is because \(\sum_{t\in [T]}G^t = \sum_{t,t'\in [T]}a^{t'} \left({T \delta^{tt'}-1}\right) = 0.\)
Additionally, summing \eqref{eq: IndividualCost_xbar_uniform} for all \(i\in\mathcal{S}\), we obtain
\begin{align*}
&\sum_{i\in \mathcal{S}}c_i(x^\dagger) = \sum_{t\in [T]}a^t \left(\frac{D_{\mathcal{S}}}{T} - \frac{G^t}{T}\sum_{i\in\mathcal{S}}\Gamma^\dagger_i\right) + \frac{bDD_{\mathcal{S}}}{T} \\& +\frac{b\Gamma^\dagger\sum_{i\in \mathcal{S}}\Gamma^\dagger_i\sum_{t\in [T]}(G^t)^2}{T^2} + 
    \frac{b}{2T^2}\sum_{t\in [T]} (G^t)^2\sum_{i\in\mathcal{S}}\frac{\mu^i}{b}(\Gamma_i^\dagger)^2 \\ 
    &= \left(\frac{AD_{\mathcal{S}}-\sum_{i\in\mathcal{S}}\Gamma^\dagger_i\sum_{t\in [T]}a^tG^t}{T} + \frac{bDD_{\mathcal{S}}}{T}\right) \\&\hspace{2cm}+ \frac{b\sum_{t\in [T]}(G^t)^2}{T^2}\bigg(\Gamma^\dagger\sum_{i\in\mathcal{S}}\Gamma^\dagger_i+\sum_{i\in\mathcal{S}}\frac{\mu^i}{2b}(\Gamma_i^\dagger)^2\bigg),
\end{align*}
where \(A:= \sum_{t\in [T]}a^t\). 
Analogously, we obtain the cost at Nash equilibrium as follows 
\begin{align*}
\sum_{i\in \mathcal{S}}c_i(x^\ast)
    &=  \left(\frac{AD_{\mathcal{S}}-\sum_{i\in\mathcal{S}}\Gamma^\ast_i\sum_{t\in [T]}a^tG^t}{T} + \frac{bDD_{\mathcal{S}}}{T}\right) \\&+ \frac{b\sum_{t\in [T]}(G^t)^2}{T^2}\bigg(\Gamma^\ast\sum_{i\in\mathcal{S}}\Gamma^\ast_i+\sum_{i\in\mathcal{S}}\frac{\mu^i}{2b}(\Gamma_i^\ast)^2\bigg).
\end{align*}
Using the above equations, we obtain that  \(\sum_{i\in\mathcal{S}}c_i(x^\dagger) > \sum_{i\in\mathcal{S}}c_i(x^\ast)\) if and only if  \(g^\dagger_{\mathcal{S}}(\mu,b) - g^{\ast}_{\mathcal{S}}(\mu,b) \geq (\Gamma^\dagger - \Gamma^\ast)h(A)\). This completes the proof.

\end{document}